\documentclass{svjour3hack}

\usepackage[utf8]{inputenc}
\usepackage{amsmath}
\usepackage{amssymb}
\usepackage{bbm}
\usepackage{tikz}
\usepackage{hyperref}
\usepackage{xparse}
\smartqed  
\makeatletter{}\newcommand{\Q}{\mathbb{Q}}
\newcommand{\R}{\mathbb{R}}
\newcommand{\Z}{\mathbb{Z}}

\renewcommand{\O}{\mathcal{O}}

\DeclareMathOperator{\rc}{rc}

\newcommand{\unitvec}{\mathbbm{e}}
\newcommand{\zerovec}{\mathbb{O}}

\newcommand{\ATSP}{\mathrm{ATSP}}
\newcommand{\STSP}{\mathrm{STSP}}

\newcommand{\CONN}{\mathrm{CONN}}
\newcommand{\BRANCH}{\mathrm{BRANCH}}
\newcommand{\ARB}{\mathrm{ARB}}
\newcommand{\SPT}{\mathrm{SPT}}
\newcommand{\FORESTS}{\mathrm{FORESTS}}
\newcommand{\DIFF}{\mathrm{DIFF}}
\newcommand{\PERM}{\mathrm{PERM}}
\newcommand{\EVEN}{\mathrm{EVEN}}
\newcommand{\ODD}{\mathrm{ODD}}
\newcommand{\TJOINS}{T\text{-}\mathrm{JOINS}}

\newcommand{\Sn}{\mathcal{S}_n}

\DeclareMathOperator{\aff}{aff}
\DeclareMathOperator{\conv}{conv}

\DeclareMathOperator{\ind}{ind}

\DeclareDocumentCommand\setdef{mo}{\ensuremath{\left\{#1\IfNoValueTF{#2}{}{:#2}\right\}}}

\spnewtheorem{customexample}{Example}{\bfseries}{\itshape}

\begin{document}

\title{Lower Bounds on the Sizes of Integer Programs Without Additional Variables\thanks{Part of this research was
funded by the German Research Foundation (DFG): ``Extended Formulations in Combinatorial Optimization'' (KA 1616/4-1)}}

\author{Volker Kaibel \and Stefan Weltge}

\institute{Otto-von-Guericke-Universit\"at~Magdeburg, Germany\\
    \email{\{kaibel,weltge\}@ovgu.de}}

\maketitle

\begin{abstract}
    Let $ X $ be the set of integer points in some polyhedron. We investigate the smallest number of facets of any
    polyhedron whose set of integer points is $ X $. This quantity, which we call the \emph{relaxation
    complexity} of $ X $, corresponds to the smallest number of linear inequalities of any integer program having $ X $
    as the set of feasible solutions that does not use auxiliary variables. We show that the use of auxiliary variables
    is essential for constructing polynomial size integer programming formulations in many relevant cases. In
    particular, we provide asymptotically tight exponential lower bounds on the relaxation complexity of the integer
    points of several well-known combinatorial polytopes, including the traveling salesman polytope and the spanning
    tree polytope.
    In addition to the material in the extended abstract \cite{KaibelW14} we include
    omitted proofs, supporting figures, discussions about properties of coefficients in such formulations, and facts
    about the complexity of formulations in more general settings.

    \end{abstract}

\makeatletter{}\section{Introduction}

For many combinatorial optimization problems, their set of feasible solutions is a certain set of subsets of a finite
ground set $ E $ and there exists a vector $ c \in \R^E $ such that the objective value of each feasible set $ F
\subseteq E $ is equal to $ \sum_{e \in F} c_e $.
A central paradigm in  combinatorial optimization is to identify each feasible subset $ F $ with its
\emph{characteristic vector} $ \chi(F) \in \{0,1\}^E $, where $ \chi(F)_e = 1 \iff e \in F $, and to consider the
(equivalent) problem of maximizing or minimizing the linear function $ x \mapsto \langle c,x \rangle $ over these
vectors.
In order to treat this problem algorithmically, one needs an algebraic description of the set $ X := \{\chi(F) : F \subseteq E \text{ feasible}\} $. The standard approach that has been followed extremely successfully for many problems is to use a system of linear
inequalities $ Ax \leq b $ with 
\begin{equation}
    \label{eq:description_relaxation}
    X = \setdef{ x \in \Z^E }[ Ax \leq b ]\,,
\end{equation}
i.e., an \emph{integer linear programming} (\emph{ILP}) formulation.  
A most prominent example of this approach is provided by the traveling salesman problem. 
Let $ K_n = (V_n,E_n) $ be the undirected complete graph on $ n $ nodes and $ \STSP_n $ the set of characteristic
vectors of hamiltonian cycles in $ K_n $.
One ILP-formulation for $ X = \STSP_n $ of type~\eqref{eq:description_relaxation}, which has established itself as
\emph{the} starting point for solving the traveling salesman problem, is the so-called \emph{subtour elimination}
formulation
\begin{align}
    \nonumber
    \STSP_n = \Big\{ x \in \{0,1\}^{E_n} :
    x(\delta(S)) & \geq 2 \quad \forall \, \emptyset \neq S \subsetneq V_n \\
    \label{eq:subtour-relaxation}
    x(\delta(v)) & = 2 \quad \forall \, v \in V_n \quad \quad \Big\}.
\end{align}
This description uses exponentially many (in $ n $) linear inequalities, but nevertheless, it can be used computationally efficiently (both with respect to theoretical and to practical aspects), because the separation problem associated with these inequalities can be solved efficiently. Though the mere size of the description thus does not harm its algorithmic  treatment, one may wonder whether an ILP-formulation of the traveling salesman problem in the above setup necessarily needs to be of exponential size, while for other combinatorial optimization problems (both NP-hard ones as  well as polynomial time solvable ones) like, e.g., the maximum clique, the maximum cut, or the maximum matching problem, polynomial size ILP-formulations are ready at hand. The work reported about in this paper had its origins in this question, which, to our initial slight  surprise, apparently has not been treated before.

Besides pure mathematical curiosity, the question for small rather than only efficiently separable ILP-formulations of combinatorial optimization problems also has some practical relevance. If simplicity of implementing a solution procedure  is a more important issue than efficiency of the solution process, a small ILP-formulation that can be fed immediately into a black box ILP-solver is likely to be preferred over one for which a separation procedure has to be implemented and linked to the solver. Therefore, also from a practical point of view there seems to be some interest in understanding better the general possibilities and limits of formulating combinatorial optimization problems via ILP's. 

Throughout this paper, given a set $ X \subseteq \Z^d $, let us call a polyhedron $ R \subseteq \R^d $ a
\emph{relaxation} for $ X $ if $ R \cap \Z^d = \conv(X) \cap \Z^d $ holds.
Furthermore, the smallest number $ \rc(X) $ of facets of any relaxation for $ X $ will be called the \emph{relaxation complexity} of $
X $.
With this notation, the initial question asks for the asymptotic behavior of $ \rc(\STSP_n) $.

Except for a paper by Jeroslow \cite{Jeroslow75}, the authors are not aware of any reference that deals with a similar
quantity.
In his paper, for a set $ X \subseteq \{0,1\}^d $ of binary vectors, Jeroslow introduces the term \emph{index} of $ X $
(short: $ \ind(X) $), which is defined as the smallest number of inequalities needed to separate $ X $ from the
remaining points in $ \{0,1\}^d $.
Thus, the notion of relaxation complexity can be seen as a natural extension of the index with respect to general
subsets of $ \Z^d $.
Clearly, we have that $ \ind(X) \leq \rc(X) $ holds for all sets $ X \in \{0,1\}^d $.
On the other hand, as we will briefly discuss in Section~\ref{sec:cube}, both quantities differ at most by an additive
term of $ d + 1 $.
As the main result in his paper, Jeroslow shows that $ 2^{d-1} $ is an upper bound on $ \ind(X) $, which is attained by the set of
binary vectors of length $ d $ that contain an even number of ones, see Section~\ref{sec:parity}.
We generalize his idea of bounding the index of a set $ X \subseteq \{0,1\}^d $ from below to provide lower bounds on
the relaxation complexity of general $ X $, see Section~\ref{sec:lower-bounds}.

This allows us to provide exponential lower bounds on the relaxation complexities of sets associated to problems that
are variants of the traveling salesman, the spanning tree, or the $T$-join problem.
In particular, we show that the asymptotic growth of $ \rc(\STSP_n) $ is $ 2^{\Theta(n)} $. In this sense,  the exponentially large subtour
elimination formulation thus is asymptotically smallest possible.

Of course, exponential lower bounds on the relaxation complexity only imply that it is impossible to come up with polynomial size ILP-formulations in the space of the variables~$x$ that are naturally associated with the respective problem in the way described above. They do not refer to ILP-formulations of the form
\begin{equation}
    \label{eq:intro-general-ips}
    \min \setdef{ \langle c,x \rangle }[Ax + By \leq b, \, x \in \Z^E, \, y \in \Z^m],
\end{equation}
where the vector $ y $ consists of additional variables and $ X = \{ x \in \Z^E : \exists y \in \Z^m: \ Ax + By \leq b
\} $.
Indeed, there exist classical formulations for $ X = \STSP_n $ for which the system $ Ax + By \leq b $ in~\eqref{eq:intro-general-ips} consists of
polynomially many linear inequalities, see, e.g. \cite{MillerTZ60} or \cite{GavishG78}.
As we briefly discuss in Section~\ref{sec:basics}, it turns out that in fact every in a certain sense reasonable  combinatorial
optimization problem  admits polynomial size descriptions of type~\eqref{eq:intro-general-ips}.

In contrast to this, recent results show that for many problems the associated polytope $\conv(X)$ has exponential extension complexity, i.e., every representation $ \conv(X) = \{ x \in \R^E : \exists y \in \R^m: \ Ax + By \leq b
\} $ has exponential size.
This refers to both NP-hard problems like the traveling salesman problem and others~\cite{FioriniMPTW12} (see also~\cite{AvisT13,PokuttaV13}) as well as even to the polynomial time solvable matching problem~\cite{Rothvoss14}. Thus, our exponential lower bounds on $ \rc(\STSP_n) $ in particular show that in polynomial size formulations of the traveling salesman problem of type~\eqref{eq:intro-general-ips} both the use of additional variables and imposing integrality constraints are necessary. 

Our paper is organized as follows:
In Section \ref{sec:basics}, we discuss basic facts about existence and properties of certain types of integer programs
describing sets that are mainly relevant in the field of combinatorial optimization.
This part gives further motivation on our notion of relaxations.
Section \ref{sec:cube} addresses to the relaxation complexity of the hypercube's vertices and serves as a simple
starting point to get familiar with questions asked (and basically being answered) in this paper.
Most of our main results are contained in Section~\ref{sec:lower-bounds}.
There, we introduce the concept of \emph{hiding sets}, which turns out to be a powerful technique to provide lower
bounds on $ \rc(X) $.
Finally, we use this technique to give exponential lower bounds on the sizes of relaxations for concrete structures that
occur in many practical IP-formulations.
In the last section of the paper, we briefly discuss some open questions regarding the rationality of minimum size
relaxations.
 
\makeatletter{}\section{Basic Observations}
\label{sec:basics}

General sets of integer points do not have to admit any relaxation.
Therefore, we will focus on \emph{polyhedral} sets $ X $, i.e., sets whose convex hull is a polyhedron.
By definition, we have that $ \rc(X) $ is finite for such sets.
Further, in this setting, it is easy to see that any relaxation corresponds to a valid IP-formulation and vice versa:
\begin{proposition}
    Let $ X \subseteq \Z^d $ be polyhedral and $ P \subseteq \R^d $ a polyhedron. Then, $ P $ is a relaxation for $ X $
    if and only if $ \sup \{ \langle c,x \rangle : x \in P \cap \Z^d \} = \sup \{ \langle c,x \rangle : x \in X \} $
    holds for all $ c \in \R^d $.
\end{proposition}
\subsection{Projection is a Powerful Tool}
Following Schrijver's proof \cite[Thm. 18.1]{Schrijver86} of the fact that integer programming is $ \mathrm{NP} $-hard,
one finds that for any language $ \mathcal{L} \subseteq \{0,1\}^* $ that is in $ \mathrm{NP} $, there is a polynomial $
p $ such
that for any $ k > 0 $ there is a system $ Ax + By \leq b $ of at most $ p(k) $ linear inequalities and $ m \leq p(k) $
auxiliary variables with
\[
    \setdef{x \in \{0,1\}^k}[x \in \mathcal{L}] = \setdef{x \in \{0,1\}^k}[\exists \, y \in \{0,1\}^m \ Ax+By \leq b].
\]
Further, suppose we are given a \emph{boolean circuit} and let $ X \subseteq \{0,1\}^d $ be the set of inputs that
evaluate to true. It is straightforward to model the outputs of all intermediate gates in terms of additional variables
and linear inequalities: For inputs $ y_1, y_2 \in \{0,1\} $, the resulting output $ y_3 $ of a, say, OR-gate is the
unique solution $ y_3 \in [0,1] $ of the system $ y_1 \leq y_3, \, y_2 \leq y_3, \, y_3 \leq y_1+y_2 $. A crucial
property of these constraints is that if the inputs have $ 0/1 $-values, then $ y_3 $ is also implicitly forced to take
its value in $ \{0,1\} $. It is straightforward to make analogous observations for AND-gates and NOT-gates, see, e.g.,
\cite[p. 445]{Yannakakis91}. Since for every language $ \mathcal{L} \in \mathrm{P} $ there exists a polynomial time
algorithm to construct (polynomial size) boolean circuits that decide $ \mathcal{L} $, see, e.g.,
\cite[pp.~109--110]{AroraB09}, we conclude:
\begin{proposition}
    \label{prop:polynomial-size-integer-programs}
    Let $ X_d \subseteq \{0,1\}^d $ be a family of sets such that the membership problem ``Given $ x \in
    \{0,1\}^d $, is $ x $ in $ X_d $?'' is in $ \mathrm{NP} $. Then there exists a polynomial $ p $ such that for any $
    d $ there is a system $ Ax + By \leq b $ of at most $ p(d) $ linear inequalities and $ m \leq p(d) $ auxiliary
    variables with
    \[
        X_d = \setdef{x \in \{0,1\}^d}[Ax + By \leq b, \, y \in \Z^m].
    \]
    If the membership problem is even in $ \mathrm{P} $, then there exist such systems without integrality
    constraints on the auxiliary variables $ y $. \qed
\end{proposition}
Coming back to the exponential size of the subtour elimination relaxation, it might be argued that the number of
inequalities is not the right measure of complexity since it is still possible to optimize linear functions over the
subtour elimination relaxation in polynomial time.
However, the mere existence of a relaxation over which optimization can be performed in polynomial time is nothing
special, as the following result shows.
\begin{proposition}
    Let $ X_d \subseteq \{0,1\}^d $ be a family of sets such that the membership problem ``Given $ x \in
    \{0,1\}^d $, is $ x $ in $ X_d $?'' is in $ \mathrm{P} $. Then there exists a family of relaxations $ R_d $ for $
    X_d $ such that linear programming over $ R_d $ can be done in polynomial time.
\end{proposition}
\begin{proof}
    By Proposition \ref{prop:polynomial-size-integer-programs}, we know that for each $ d $ there exists a system $ Ax +
    By \leq b $ of polynomially many linear inequalities such that $ X_d = \{ x \in \{0,1\}^d : Ax + By \leq b, \, y \in
    \R^m \} $. As mentioned in the above argumentation, such systems even can be constructed by a
    polynomial time algorithm. Thus, setting
    \begin{align*}
        R'_d &:= \{ (x,y) : Ax + By \leq b, \, x \in [0,1]^d, \, y \in \R^m \} \\
        R_d & := \{ x \in \R^d : \exists y \in \R^m \, (x,y) \in R'_d \}
    \end{align*}
    gives us the desired relaxations $ R_d $. Indeed, given $ c \in \Q^d $ we have that
    \[
        \max \, \{ \langle c,x \rangle : x \in R_d \} = \max \, \{ \langle c,x \rangle : (x,y) \in R'_d \},
    \]
    where the latter problem can be solved in time polynomially bounded in $ d $ and the encoding length of $ c $. \qed
\end{proof}
\subsection{Encoding Lengths of Coefficients}
Our notion of relaxation complexity does not involve sizes of coefficients in (minimum size) relaxations.
In fact, we are not aware of any bounds on the coefficients' sizes of minimum size relaxations for general polyhedral
sets. We even do not know if they can be chosen to be rational in every case, see Section \ref{sec:rationality} for a
discussion. However, let us consider the following observation concerning binary vectors:
\begin{proposition}
    \label{prop:encoding-lengths-coefficients}
    There exists a polynomial $ p $ such that for any real vector $ a \in \R^d $ and any real number $ \gamma \in \R $
    there is a rational vector $ a' \in \Q^d $ and a rational number $ \gamma' \in \Q^d $ satisfying
    \begin{itemize}
        \item $ \{ x \in \{0,1\}^d : \langle a,x \rangle \leq \gamma \} = \{ x \in \{0,1\}^d : \langle a',x \rangle \leq
        \gamma' \} $ and
        \item the encoding lengths of $ a' $ and $ \gamma' $ are both bounded by $ p(d) $.
    \end{itemize}
\end{proposition}
\begin{proof}
    We may assume that $ X := \{ x \in \{0,1\}^d : \langle a,x \rangle \leq \gamma \} $ is not empty. By setting
    $ \bar{x} := \arg \max \{ \langle a,x \rangle : x \in X \} $, let us define the affine isomorphism $ \varphi \colon
    \R^d \to \R^d $ via $ \varphi(x) := x - \bar{x} $. Clearly, we now have that
    \[
        \langle a,x \rangle \leq \gamma \iff \langle a,\varphi(x) \rangle \leq 0
    \]
    holds for all $ x \in \{0,1\}^d $. Thus, the polyhedron
    \[
        P := \Big\{ \tilde{a} \in \R^d : \langle y,\tilde{a} \rangle \leq 0 \ \forall \, y \in \varphi(X), \,
        \langle y,\tilde{a} \rangle \geq 1 \ \forall \, y \in \varphi(\{0,1\}^d \setminus X) \Big\}
    \]
    is not empty. Since $ P $ can be described by a system of linear inequalities whose coefficients are in $ \{-1,0,1\}
    $, we know that there exists a rational point $ a' \in P $ whose encoding length is polynomially bounded in $ d
    $, see \cite{Schrijver86}.

    For any point $ x \in \{0,1\}^d $, by the definition of $ a' $, we now have that
    \begin{align*}
        \langle a,x \rangle \leq \gamma
        \iff \langle a,\varphi(x) \rangle \leq 0
        & \iff \langle a',\varphi(x) \rangle \leq 0 \\
        & \iff \langle a',x-\bar{x} \rangle \leq 0
        \iff \langle a',x \rangle \leq \langle a',\bar{x} \rangle.
    \end{align*}
    and hence setting $ \gamma' := \langle a',\bar{x} \rangle $ completes the proof. \qed
\end{proof}
Proposition \ref{prop:encoding-lengths-coefficients} tells us that if we are given a set $ X \subseteq \{0,1\}^d $, then we
can always find a rational relaxation $ \tilde{R} $ for $ X $ whose number of facets is close to $ \rc(X) $ and the
encoding lengths of coefficients in a suitable outer description of $ \tilde{R} $ can be polynomially bounded in $ d $.
Indeed, let $ R $ be a minimal relaxation for $ X $, perturb its facet-defining inequalities according to Proposition
\ref{prop:encoding-lengths-coefficients} and call the obtained polyhedron $ R' $. If we now choose $ C \subseteq \R^d $
to be any relaxation for $ \{0,1\}^d $, we obtain that $ \tilde{R} := R' \cap C $ is a relaxation for $ X $. In Section
\ref{sec:cube}, we will see that the number of inequalities we thus have to add, i.e., the number of facets of $ C
$, can be assumed to be at most $ d+1 $.
\subsection{Restricting to Facet-Defining Inequalities}
At the end of this section, let us briefly discuss a further requirement on relaxations:
Many known relaxations for sets $ X \subseteq \Z^d $ that are identified with feasible points in combinatorial problems
are defined by linear inequalities of which, preferably, most of them are facet-defining for $ \conv(X) $. Clearly, this
has important practical reasons since such formulations are tightest possible in some sense. However, if one is
interested in a relaxation that has as few number of facets as possible, one cannot only use
facet-defining inequalities of $ \conv(X) $: In fact, in the next section we will see that $ \rc(\{0,1\}^d) = d + 1 $
whereas by
removing any of the cube's inequalities the remaining polyhedron gets unbounded. Nevertheless, the restriction turns out
to be not too hard:
\begin{proposition}
    Let $ X \subseteq \Z^d $ be polyhedral and $ \rc_F(X) $ the smallest number of facets of any relaxation for $ X $
    whose facet-defining inequalities are also facet-defining for $ \conv(X) $. Then, $ \rc_F(X) \leq \dim(X) \cdot
    \rc(X) $.
\end{proposition}
\begin{proof}
    By Carath\'eodory's Theorem, any facet-defining inequality of a relaxation $ R $ for $
    X $ can be replaced by $ \dim(X) $ many facet-defining inequalities of $ \conv(X) $. The resulting polyhedron is
    still a relaxation for $ X $. \qed
\end{proof}
 
\makeatletter{}\section{Warm-Up: The Cube}
\label{sec:cube}

\noindent
As mentioned in the introduction, Jeroslow \cite{Jeroslow75} showed that for any set $ X \subseteq \{0,1\}^d $, one
needs at most $ 2^{d-1} $ many linear inequalities in order to separate $ X $ from $ \{0,1\}^d \setminus X $. If $ P
\subseteq \R^d $ is a polyhedron such that $ P \cap \{0,1\}^d = X $, then, in order to construct a relaxation for $ X $,
we need to additionally separate all points $ \Z^d \setminus \{0,1\}^d $ from $ X $.  This can be done by intersecting $
P $ with a relaxation for $ \{0,1\}^d $. We conclude:

\begin{proposition}
    \label{prop:binary}
    Let $ X \subseteq \{0,1\}^d $. Then $ \rc(X) \leq 2^{d-1} + \rc(\{0,1\}^d) $. \qed
\end{proposition}

\noindent
Motivated by Proposition \ref{prop:binary}, we are interested in the relaxation complexity of $ \{0,1\}^d $. Since $
[0,1]^d = \conv(\{0,1\}^d) $, we obviously have that $ \rc(\{0,1\}^d) \leq 2d $. However, it turns out that one can
construct a relaxation of only $ d+1 $ facets:

\begin{lemma}
    \label{lem:rc-cube}
    For $ d \geq 1 $, we have
    \[
        \{0,1\}^d = \setdef{x \in \Z^d}[x_k \leq 1 + \sum_{i=k+1}^d 2^{-i} x_i \ \forall \, k \in [d], \,
        x_1 + \sum_{i=2}^d 2^{-i} x_i \geq 0].
    \]
\end{lemma}
\begin{proof}
    Obviously, any point $ x \in \{0,1\}^d $ satisfies
    \begin{equation}
        \label{eq:rc-cube-1}
        x_k \leq 1 + \sum_{i=k+1}^d 2^{-i} x_i
    \end{equation}
    for all $ k \in [d] $ and
    \begin{equation}
        \label{eq:rc-cube-2}
        x_1 + \sum_{i=2}^d 2^{-i} x_i \geq 0.
    \end{equation}
    Let $ x \in \Z^d $ be any integer point that satisfies \eqref{eq:rc-cube-1} for all $ k \in [d] $ as well as
    \eqref{eq:rc-cube-2}. First, we claim that $ x_i \leq 1 $ for all $ i \in [d] $: Suppose that $ x_k > 1 $ for some $
    k \in [d] $. W.l.o.g. we may assume that $ x_i \leq 1 $ for all $ i > k $ and obtain
    \[
        x_k \stackrel{\text{\eqref{eq:rc-cube-1}}}{\leq}
        1 + \sum_{i=k+1}^d 2^{-i} x_i \leq 1 + \sum_{i=k+1}^d 2^{-i} < 2,
    \]
    a contradiction. Further, we see that $ x_1 \geq 0 $ since (due to $ x_i \leq 1 $ for all $ i $)
    \[
        x_1 \stackrel{\text{\eqref{eq:rc-cube-2}}}{\geq}
        - \sum_{i=2}^d 2^{-i} x_i \geq - \sum_{i=2}^d 2^{-i} > -1.
    \]
    It remains to show that $ x_i \geq 0 $ for all $ i \in [d] \setminus \{1\} $. Suppose that $ x_j \leq -1 $ for some
    $ j \in [d] \setminus \{1\} $ and $ x_i \geq 0 $ for all $ i < j $. In this case, we claim that $ x_i = 0 $ for all
    $ i < j $: Otherwise, let $ k $ be the largest $ k<j $ such that $ x_k > 0 $ (and hence $ x_k = 1 $). By inequality
    \eqref{eq:rc-cube-1}, we would obtain
    \begin{align*}
        1 = x_k \leq 1 + \sum_{i=k+1}^d 2^{-i} x_i & = 1 + 2^{-j} x_j + \sum_{i=j+1}^d 2^{-i} x_i \\
        & \leq 1 + 2^{-j} \cdot (-1) + \sum_{i=j+1}^d 2^{-i} < 1.
    \end{align*}
    Thus, we have $ x_i \geq 0 $ for all $ i < j $, and hence, by inequality \eqref{eq:rc-cube-2}, we deduce
    \[
        0 \leq x_1 + \sum_{i=2}^d 2^{-i} x_i = 2^{-j} x_j + \sum_{i=j+1}^d 2^{-i} x_i \leq
        2^{-j} \cdot (-1) + \sum_{i=j+1}^d 2^{-i} < 0,
    \]
    a contradiction. \qed
\end{proof}

\noindent
To show that this construction is best possible, note that if a polyhedron that contains $ \{0,1\}^d $ (and hence is $ d
$-dimensional) has less than $ d+1 $ facets, it must be unbounded. In order to show that such a (possibly irrational)
polyhedron must contain infinitely many integer points (and hence cannot be a relaxation of $ \{0,1\}^d $), we make use
of Minkowski's theorem:

\begin{theorem}[\textsc{Minkowski} \cite{Minkowski96}]
    \label{thm:minkowski}
    Any convex set which is symmetric with respect to the origin and with volume greater than $ 2^d $ contains a
    non-zero integer point.
\end{theorem}

\noindent
For $ \varepsilon > 0 $ let $ B_{\varepsilon} := \setdef{x \in \R^d}[\|x\|_2 < \varepsilon] $ be the open ball with radius
$ \varepsilon $. As a direct consequence of Minkowski's theorem, the following corollary is useful for our
argumentation.

\begin{corollary}
    \label{cor:minkowski}
    Let $ c \in \R^d \setminus \{ \zerovec \} $, $ \lambda_0 \in \R $ and $ \varepsilon > 0 $. Then
    \[
        L(c,\lambda_0,\varepsilon) := \setdef{\lambda c \in \R^d}[\lambda \geq \lambda_0] + B_{\varepsilon}
    \]
    contains infinitely many integer points.
\end{corollary}
\begin{proof}
    Let us define $ L(c, \varepsilon) := \setdef{\lambda c \in \R^d}[\lambda \in \R] + B_{\varepsilon} $. Clearly,
    $ L(c, \varepsilon) $ is symmetric with respect to the origin.
    Since $ L(c,0,\varepsilon) \setminus L(c,\lambda_0,\varepsilon) $ is bounded,
    it suffices to show that $ L(c, \varepsilon) $ contains infinitely many integer points.

    Since the latter statement is obviously true if
    \begin{equation}
        \label{eq:nointegerpoints}
        \setdef{ \lambda c }[\lambda \in \R] \cap \Z^d \neq \{ \zerovec \},
    \end{equation}
    we assume that \eqref{eq:nointegerpoints} does not hold. Setting $ \varepsilon_1 := \varepsilon $, by Theorem
    \ref{thm:minkowski}, $ L(c,\varepsilon_1) $ contains a point $ p_1 \in \Z^d \setminus \{ \zerovec \} $. Since
    \eqref{eq:nointegerpoints} does not hold, there exists some $ \varepsilon_2 > 0 $ such that $ L(c,\varepsilon_2)
    \subseteq L(c,\varepsilon_1) $ and $ p_1 \notin L(c,\varepsilon_2) $. Again, by Theorem \ref{thm:minkowski}, $
    L(c,\varepsilon_2) $ also contains a point $ p_2 \in \Z^d \setminus \{ \zerovec \} $. Further, there is also some $
    \varepsilon_3 > 0 $ such that $ L(c,\varepsilon_3) \subseteq L(c,\varepsilon_2) $ and $ p_2 \notin
    L(c,\varepsilon_3) $.  By iterating these arguments, we obtain an infinite sequence $ (\varepsilon_i, p_i) $ such
    that $ p_i \in L(c,\varepsilon_i) \cap \Z^d \subseteq L(c,\varepsilon) \cap \Z^d $ and $ p_i \notin
    L(c,\varepsilon_{i+1}) $ for all $ i > 0 $. In particular, all $ p_i $ are distinct. \qed
\end{proof}

\begin{theorem}
    For $ d \geq 1 $, we have that $ \rc(\{0,1\}^d) = d + 1 $.
\end{theorem}
\begin{proof}
    By Lemma \ref{lem:rc-cube}, we already know that $ \rc(\{0,1\}^d) \leq d + 1 $. Suppose there is a relaxation $ R
    \subseteq \R^d $ for $ \{0,1\}^d $ with less than $ d + 1 $ facets. As mentioned above, since $ \dim(R) \geq
    \dim(\{0,1\}^d) = d $, $ R $ has to be unbounded.

    By induction over $ d \geq 1 $, we will show that any unbounded polyhedron $ R \subseteq \R^d $ with $ \{0,1\}^d
    \subseteq R $ contains infinitely many integer points. Hence, it cannot be a relaxation of $ \{0,1\}^d $. Clearly,
    our claim is true for $ d = 1 $. For $ d \geq 1 $, let $ c \in \R^d \setminus \{ \zerovec \} $ be a direction such
    that $ x + \lambda c \in R $ for any $ x \in R $ and $ \lambda \geq 0 $. Since $ \{0,1\}^d $ is invariant under
    affine maps that map a subset of coordinates $ x_i $ to $ 1-x_i $, we may assume that $ c \geq \zerovec $.

    If $ c > \zerovec $, then there is some $ \lambda_0 > 0 $ such that $ \lambda_0 c $ is in the interior of $ [0,1]^d
    $. Thus,
    there is some $ \varepsilon > 0 $ such that $ \lambda_0 c + B_{\varepsilon} \subseteq [0,1]^d \subseteq R $. By the
    definition of $ c $ and $ \varepsilon $, we thus obtained that $ L(c,\lambda_0,\varepsilon) \subseteq R $. By
    Corollary \ref{cor:minkowski}, it follows that $ L(c,\lambda_0) $ contains infinitely many integer points and so
    does $ R $.

    Otherwise, we may assume that $ c_d = 0 $. Let $ \mathcal{H}_d := \{x \in \R^d : x_d = 0 \} $ and $ p \colon
    \mathcal{H}_d \to \R^{d-1} $ be the projection onto the first $ d-1 $ coordinates. Then, the polyhedron $ \tilde{R}
    = p(R \cap \mathcal{H}_d) $ is still unbounded  and contains $ \{0,1\}^{d-1} = p(\{0,1\}^d) $. By induction, $
    \tilde{R} $ contains
    infinitely many integer points and so does $ R $. \qed
\end{proof}

\noindent
With Proposition \ref{prop:binary} we thus obtain:

\begin{corollary}
    \label{cor:binary}
    Let $ X \subseteq \{0,1\}^d $. Then $ \rc(X) \leq 2^{d-1} + d + 1 $. \qed
\end{corollary}
 
\makeatletter{}\section{Lower Bounds}
\label{sec:lower-bounds}

\noindent
The technique used to obtain a lower bound on the relaxation complexity of the cube is apparently useless to prove
exponential lower bounds for $ \rc(\STSP_n) $ since it only provides a lower bound of at most $ d + 1$.
Therefore, let us introduce another simple framework to provide lower
bounds on the relaxation complexity for polyhedral sets $ X \subseteq \Z^d $.

\begin{definition}
    Let $ X \subseteq \Z^d $. A set $ H \subseteq \aff(X) \cap \Z^d \setminus \conv (X) $ is called a \emph{hiding set} for $ X
    $ if for any two distinct points $ a,b \in H $ we have that $ \conv \{a,b\} \cap \conv (X) \neq \emptyset $.
\end{definition}

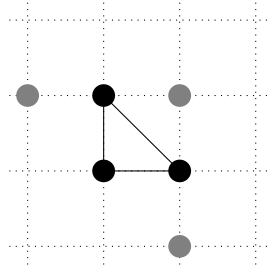
\begin{figure}
    \begin{center}
        \makeatletter{}\def\radius{0.15}

\begin{tikzpicture}
    \tikzstyle{simp} = [fill=black];
    \tikzstyle{hid} = [fill=black!50];

        \clip (-1.25,-1.25) rectangle (2.25,2.25);
    \draw[dotted] (-2,-2) grid (3,3);

        \draw (0,0) -- (1,0) -- (0,1) -- cycle;
    \fill[simp] (0,0) circle (\radius);
    \fill[simp] (1,0) circle (\radius);
    \fill[simp] (0,1) circle (\radius);

        \fill[hid] (1,1) circle (\radius);
    \fill[hid] (-1,1) circle (\radius);
    \fill[hid] (1,-1) circle (\radius);
\end{tikzpicture}
 
    \end{center}
    \caption{Hiding set (gray) for the vertices of the standard 2-simplex (black).}
    \label{fig:hiding-set-simplex}
\end{figure}

\begin{proposition}
    \label{prop:hiding-set}
    Let $ X \subseteq \Z^d $ be polyhedral and $ H \subseteq \aff(X) \cap \Z^d \setminus X $ a hiding set for $ X $.
    Then, $ \rc(X) \geq |H| $.
\end{proposition}
\begin{proof}
    Let $ R \subseteq \R^d $ be a relaxation for $ X $. Since $ H \subseteq \aff(X) \subseteq \aff(R) $, any point in $
    H $ must be separated from $ X $ by a facet-defining inequality of $ R $. 

    Suppose that a facet-defining inequality $ \langle \alpha,x \rangle \leq \beta $ of $ R $ is
    violated by two distinct points $ a,b \in H $. Since $ H $ is a hiding set, there exists a point $ x \in \conv
    \{a,b\} \cap \conv (X) $. Clearly, $ x $ does also violate $ \langle \alpha,x \rangle \leq \beta $, which is a
    contradiction since $ \langle \alpha,x \rangle \leq \beta $ is valid for $ R \supseteq \conv (X) $.

    Thus, any facet-defining inequality of $ R $ is violated by at most one point in $ H $. Hence, $ R $ has at least $
    |H| $ facets. \qed
\end{proof}

\noindent
Let $ \Delta_d := \{ \zerovec, \unitvec_1, \dotsc, \unitvec_d \} \subseteq \{0,1\}^d $ be the vertices of the standard $
d $-simplex. As an example, see Fig. \ref{fig:hiding-set-simplex} illustrating a hiding set for $ \Delta_2 $, which
yields a simple proof of the fact that any relaxation for these points must have at least three facets. Unfortunately,
it turns out that we cannot construct larger hiding sets for any $ \Delta_d $.

\begin{proposition}
    Any hiding set for $ \Delta_d $ has cardinality at most $ 3 $.
\end{proposition}
\begin{proof}
    Let $ H $ be any hiding set for $ \Delta_d $. Then, for any of the inequalities $ x_i \geq 0 $, $ i \in [d] $ as
    well as $ \sum_{i=1}^d x_i \leq 1 $, there exists at most one point in $ H $ violating it. In particular, at most
    one of the points in $ H $ lies in the nonnegative orthant.

    Let us assume that $ |H| \geq 4 $. Then there are distinct points $ a, b, p, q \in H $ with $ a_i < 0 $ and $ b_j <
    0 $ for some $ i, j \in [d] $ with $ i \neq j $. Since $ \lambda a + (1 - \lambda) p \in \Delta \subseteq \R^d_+ $ for
    some $ \lambda \in [0,1]^d $, we must have $ p_i > 0 $. Since $ p_i \in \Z $, it follows that $ p_i \geq 1 $.
    Analogously, we obtain $ p_j, q_i, q_j \geq 1 $.

    Consider now any point $ y = \lambda p + (1-\lambda) q \in \conv \{p,q\} $ with $ \lambda \in [0,1] $.  Note that $
    x_i + x_j \leq 1 $ is a valid inequality for all points $ x \in \conv \Delta_d $. But since $ p_i, p_j, q_i, q_j
    \geq 1 $, it is easy to see that $ y_i + y_j \geq 2 $ and hence $ y \notin \conv \Delta_d $. Thus, we obtain that $
    \conv \{p,q\} \cap \conv \Delta_d = \emptyset $, a contradiction to $ H $ being a hiding set for $ \Delta_d $. \qed
\end{proof}
This observation shows that the hiding set bound has its limitations. (As a consequence of Proposition
\ref{prop:lower-bound-simplex}, we will see that $ \rc(\Delta_d) $ indeed grows in $ d $.)

Nevertheless, in what follows, we will demonstrate that this bound is a powerful tool to provide exponential lower
bounds on the relaxation complexities of numerous interesting sets $ X $. By dividing these sets into three
classes, we try to identify general structures that are hard to model in the context of relaxations.

\makeatletter{}\subsection{Connectivity and Acyclicity}

\noindent
In many IP-formulations for practical applications, the feasible solutions are subgraphs that are required to be  
connected or
acyclic. Quite often in these cases, 
there are polynomial size IP-formulations that use auxiliary variables. For instance, for the \emph{spanning tree polytope}  there  are even 
polynomial size extended formulations \cite{Martin91}  that can be adapted to also
work for the \emph{connector polytope} $ \CONN_n $ (see below). 
In contrast, we give exponential lower bounds on the relaxation complexities of
some important representatives of this structural class.

\makeatletter{}\subsubsection{STSP \& ATSP}

As a first application of the hiding set bound, we will show that the subtour relaxation for $ \STSP_n $ has indeed
asymptotically smallest size (in the exponential sense), i.e., that $ \rc(\STSP_n) = 2^{\Theta(n)} $ holds.
In fact, we will also give an exponential lower
bound for the directed version $ \ATSP_n \subseteq \{0,1\}^{A_n} $, which is the set of characteristic vectors of
directed hamiltonian cycles in the complete directed graph on $ n $ nodes whose arcs we denote by $ A_n $.

We will first construct a large hiding set for $ \ATSP_n $. Towards this end,
let $ n = 2(N+1) $ for some integer $ N \ge 0 $ and let us consider the complete directed graph on the node set
\[
    V := \setdef{ v_1, \dotsc, v_{N+1}, w_1, \dotsc, w_{N+1} }
\]
of cardinality~$n$.
For a binary vector $ b \in \{0,1\}^N $ let us further define the arc set
\begin{align*}
    \mathcal{E}_b := & \big\{ (v_{N+1},v_1), \, (w_{N+1},w_1) \big\} \\
    & \cup \bigcup_{i:b_i = 0} \big\{ ( v_i,v_{i+1} ), \, ( w_i,w_{i+1} ) \big\} \ \cup \
    \bigcup_{i:b_i = 1} \big\{ ( v_i,w_{i+1} ), \, ( w_i,v_{i+1} ) \big\},
\end{align*}
see Fig. \ref{fig:hamburger} for an example.
Note that $ \mathcal{E}_b $ is a directed hamiltonian cycle on the node set $ V $ if and only if $ \sum_{i=1}^N b_i $ is
odd.
Thus,  the set
\[
    H_N := \setdef{ \chi(\mathcal{E}_b) }[b \in \{0,1\}^N, \, \sum_{i=1}^N b_i \text{ is even}]
\]
is clearly disjoint from $ \ATSP_{2(N+1)} $.

In this section, we will only consider graphs on these $ 2(N+1) $ nodes.
It is easy to transfer the following observations to complete graphs with an odd number of nodes 
by replacing arc $ ( v_{N+1}, v_1 ) $ in $ \mathcal{E}_b $ by a directed path including one additional node.
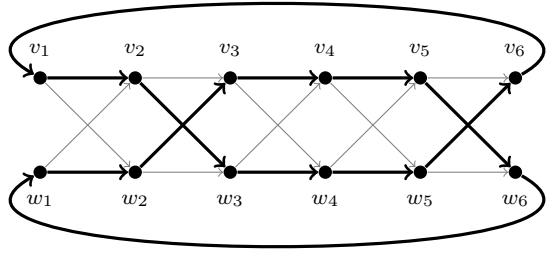
\begin{figure}
    \begin{center}
        \makeatletter{}\begin{tikzpicture}[scale=1.25]
    \tikzstyle{vertex} = [fill,circle,inner sep=0pt,minimum size=5pt];
    \tikzstyle{arc} = [color=gray,arrows=->];
    \tikzstyle{usedarc} = [very thick,arrows=->];

        \foreach \i in {1,...,6}{
        \node[vertex] (v\i) at (\i,1) {};
        \node[vertex] (w\i) at (\i,0) {};
    }

        \foreach \i in {1,...,5}{
        \pgfmathtruncatemacro{\iplusone}{\i + 1}
        \draw[arc] (v\i) -- (v\iplusone);
        \draw[arc] (w\i) -- (w\iplusone);
        \draw[arc] (v\i) -- (w\iplusone);
        \draw[arc] (w\i) -- (v\iplusone);
    }

        \draw[usedarc] (v6) to[out=30,in=150] (v1);
    \draw[usedarc] (w6) to[out=-30,in=-150] (w1);
    \draw[usedarc] (v1) -- (v2);
    \draw[usedarc] (w1) -- (w2);
    \draw[usedarc] (v2) -- (w3);
    \draw[usedarc] (w2) -- (v3);
    \draw[usedarc] (v3) -- (v4);
    \draw[usedarc] (w3) -- (w4);
    \draw[usedarc] (v4) -- (v5);
    \draw[usedarc] (w4) -- (w5);
    \draw[usedarc] (v5) -- (w6);
    \draw[usedarc] (w5) -- (v6);

        \foreach \i in {1,...,6}{
        \node at (\i,1.3) {$v_\i$};
        \node at (\i,-0.3) {$w_\i$};
    }
\end{tikzpicture}
 
    \end{center}
    \caption{Construction of the set $ \mathcal{E}_b $ for $ b = (0,1,0,0,1) $.}
    \label{fig:hamburger}
\end{figure}
\begin{lemma}
    \label{lem:hiding-set-atsp}
    $ H_N $ is a hiding set for $ \ATSP_{2(N+1)} $.
\end{lemma}
\begin{proof}
    First, note that
    \[
        H_N \subseteq \aff(\ATSP_{2(N+1)}) = \setdef{x \in \R^A}[x(\delta^{\mathrm{in}}(v)) = x(\delta^{\mathrm{out}}(v))
        = 1, \ \forall \, v \in V ]
    \]
    holds,
    where $ A $ is the set of arcs in the complete directed graph on~$V$. Let $ b, b' \in
    \{0,1\}^N $ be distinct with $ \sum_{i=1}^N b_i $ even and $ \sum_{i=1}^N b'_i $ even.
    Let $ j \in [N] $ be an index with $ b_j \neq b'_j $.
    Consider the binary vectors $ c := b + (1 - b_j) \unitvec_j $ and $ c' := b' + (1 - b'_j) \unitvec_j $.
    Clearly, we have that $ \sum_{i=1}^N c_i $ is odd and $ \sum_{i=1}^N c'_i $ is odd and hence $ \chi(\mathcal{E}_c) $
    and $ \chi(\mathcal{E}_{c'}) $ are both contained in $ \ATSP_{2(N+1)} $.
    Finally, it is easy to check that
    \[
        \chi(\mathcal{E}_b) + \chi(\mathcal{E}_{b'}) = \chi(\mathcal{E}_c) + \chi(\mathcal{E}_{c'})
    \]
    holds and hence $ \conv(\{ \chi(\mathcal{E}_b), \chi(\mathcal{E}_{b'})\} ) \cap \conv(\ATSP_{2(N+1)}) \neq \emptyset $.
\end{proof}

\begin{theorem}
    The asymptotic growth of $ \rc(\ATSP_n) $ and $ \rc(\STSP_n) $ is $ 2^{\Theta(n)} $.
\end{theorem}
\begin{proof}
    Lemma \ref{lem:hiding-set-atsp} shows that $ H_N $ is a hiding set for $ \ATSP_n $. By replacing all directed arcs
    with their undirected versions, the set $ H_N $  yields a hiding set for $ \STSP_n $. By Proposition
    \ref{prop:hiding-set}, we obtain a lower bound of $ |H_N| = 2^{\Omega(n)} $ for $ \rc(\ATSP_n) $ and $ \rc(\STSP_n)
    $. To complete the argumentation, note that both $ \ATSP_n $ and $ \STSP_n $ have relaxations of size $
    2^{\Theta(n)} $, which are variants of the formulation in \eqref{eq:subtour-relaxation}. \qed
\end{proof}
 
\makeatletter{}\subsubsection{Connected Sets}
Let $ \CONN_n $ be the set of all characteristic vectors of edge sets that form a connected spanning subgraph in the complete
graph on $ n $ nodes. The polytope
\[
    \setdef{ x \in [0,1]^{E_n} }[ x(\delta(S)) \geq 1 \ \forall \, \emptyset \neq S \subsetneq V_n ]
\]
is a relaxation for $ \CONN_n $. Thus, we have that $ \rc(\CONN_n) \leq \O(2^n) $.

For a lower bound, consider again the undirected version of our set $ H_N $. Since each point in $ H_N $ belongs to a
node-disjoint union of two cycles, we have that $ H_N \cap \CONN_n = \emptyset $. Further, we know that for any $ a, b
\in H_N $ we have that
\[
    \emptyset \neq \conv \{a,b\} \cap \conv (\STSP_n) \subseteq \conv \{a,b\} \cap \conv (\CONN_n)
\]
and since $ H_N \subseteq \aff(\CONN_n) = \R^{E_n} $, we see that $ H_N $ is also a hiding set for $ \CONN_n $. We
obtain:

\begin{corollary}
    The asymptotic growth of $ \rc(\CONN_n) $ is $ 2^{\Theta(n)} $. 
\end{corollary}
 
\makeatletter{}\subsubsection{Branchings and Forests}

\noindent
Besides connectivity, we show that, in general, it is also hard to force acyclicity in the context of relaxation. Let
therefore $ \ARB_n $ ($ \SPT_n $) be the set of characteristic vectors of arborescences (spanning trees) in the complete
directed (undirected) graph. 

\begin{theorem}
    \label{thm:arb-spt}
    The asymptotic growth of $ \rc(\ARB_n) $ and $ \rc(\SPT_n) $ is $ 2^{\Theta(n)} $.
\end{theorem}
\begin{proof}
    First, note that both the \emph{arborescence polytope} and the spanning tree polytope (i.e., $ \conv(ARB_n) $ and $
    \conv(\SPT_n) $) have $ \O(2^n) $ facets \cite{Schrijver03} and hence we have an upper bound of $ \O(2^n) $ for both
    $ \rc(\ARB_n) $ and $ \rc(\SPT_n) $.

    For a lower bound, let us modify the definition of $ \mathcal{E}_b $ by removing arc $ (w_{N+1}, w_1) $.
    Then, if $ b \in \{0,1\}^N $ with $ \sum_{i=1}^N b_i $ even, we have that $ \mathcal{E}_b $ is a
    node-disjoint union of a cycle and a path and hence not an arborescence. By following the proof of
    Lemma \ref{lem:hiding-set-atsp},
    we still have
    \[
        \chi(\mathcal{E}_b) + \chi(\mathcal{E}_{b'}) = \chi(\mathcal{E}_c) + \chi(\mathcal{E}_{c'}),
    \]
    where $ \mathcal{E}_c $ and $ \mathcal{E}_{c'} $ are spanning arborescences. (Actually, they are in fact directed paths
    visiting each node.) Since $ \aff(\ARB_n) = \R^{A_n} $, we therefore obtain that the modified set $ H_N $ is a
    hiding set for $ \ARB_n $. By undirecting all arcs, $ H_N $ also yields a hiding set for $ \SPT_n $.

    Again, by Proposition \ref{prop:hiding-set}, we deduce a lower bound of $ |H_N| = 2^{\Omega(n)} $ for both $
    \rc(\ARB_n) $ and $ \rc(\SPT_n) $. \enspace \qed
\end{proof}
\begin{remark}
    Since in the proof of Theorem \ref{thm:arb-spt} $ T_1 $ and $ T_2 $ are rooted at node $ v_1' $, the statements
    even hold if the sets $ \ARB_n $ and $ \SPT_n $ are restricted to characteristic vectors of arborescences/trees
    rooted at a fixed node.
\end{remark}

\noindent
Let $ \BRANCH_n $ ($ \FORESTS_n $) be the set of characteristic vectors of branchings (forests) in the complete directed
(undirected) graph.

\begin{corollary}
    The asymptotic growth of $ \rc(\BRANCH_n) $ and $ \rc(\FORESTS_n) $ is $ 2^{\Theta(n)} $.
\end{corollary}
\begin{proof}
    The claim follows from Theorem \ref{thm:arb-spt} and the facts that
    \begin{align*}
        \ARB_n &= \BRANCH_n \cap \Big\{ x \in \R^{A_n} : \sum_{a \in \R^{A_n}} x_a = n-1 \Big\} \\
        \SPT_n &= \FORESTS_n \cap \Big\{ x \in \R^{E_n} : \sum_{e \in \R^{E_n}} x_e = n-1 \Big\}. \enspace \qed
    \end{align*}
\end{proof}

\subsection{Distinctness}

\noindent
Another common component of practical IP-formulations is the requirement of distinctness of a certain set of vectors or
variables. Here, we consider two general cases in which we can also show that the benefit of auxiliary variables is
essential.

\makeatletter{}\subsubsection{Binary All-Different}

\noindent
In the case of the \emph{binary all-different} constraint, one requires the distinctness of rows of a binary matrix with
$ m $ rows and $ n $ columns. The set of feasible points is therefore defined by

\[
    \DIFF_{m,n} := \setdef{ x \in \{0,1\}^{m \times n} }[ x \text{ has pairwise distinct rows}].
\]

\noindent
As an example, \cite{LeeM07} give IP-formulations to solve the coloring problem in which they binary
encode the color classes assigned to each node. As a consequence, certain sets of encoding vectors have to be distinct.

By separating each possible pair of equal rows by one inequality, it is further easy to give a relaxation for $ \DIFF_{m,n} $
that has at most $ \binom{m}{2} 2^n + 2mn $ facets. In the case of $ m = 2 $, for instance, this bound turns out to be
almost tight:

\begin{theorem}
    For all $ n \geq 1 $, we have that $ \rc(\DIFF_{2,n}) \geq 2^n $.
\end{theorem}
\begin{proof}
    Let us consider the set
    \[
        H_{2,n} := \setdef{ (x,x)^T \in \{0,1\}^{2 \times n}}[x \in \{0,1\}^n].
    \]
    For $ x, y \in \{0,1\}^n $ distinct, we obviously have that
    \[
        \frac{1}{2} \left( (x,x)^T + (y,y)^T \right) = \frac{1}{2} \left( (x,y)^T + (y,x)^T \right) \in \conv(\DIFF_{2,n}).
    \]
    Since $ H_{2,n} \cap \DIFF_{2,n} = \emptyset $ and $ H_{2,n} \subseteq \aff(\DIFF_{2,n}) = \R^{2 \times n} $, $
    H_{2,n} $ is a hiding set for $ \DIFF_{2,n} $ and by Proposition \ref{prop:hiding-set} we obtain that $
    \rc(\DIFF_{2,n}) \geq |H_{2,n}| = 2^n $. \enspace \qed
\end{proof}
 
\makeatletter{}\subsubsection{Permutahedron}

\noindent
As a case in which one does not require the distinctness of binary vectors but of a set of numbers let us consider the
set
\[
    \PERM_n := \setdef{ (\pi(1),\dotsc,\pi(n)) \in \Z^n}[\pi \in \Sn],
\]
which is the vertex set of the \emph{permutahedron} $ \conv(\PERM_n) $. Rado \cite{Rado52} showed that the permutahedron
can be described via
\begin{align}
    \nonumber
    \conv(\PERM_n) = \bigg\{ x \in \R^n : \sum_{i=1}^n x_i &= \frac{n(n+1)}{2} \\
    \nonumber
    \sum_{i \in S} x_i & \geq \frac{|S| (|S|+1)}{2} \ \text{for all } \emptyset \neq S \subset [n] \\
    \label{eq:perm-ineq}
    x &\geq \zerovec \bigg\}
\end{align}
\noindent
and hence has $ \O(2^n) $ facets. Apart from that, it is a good example for a polytope having many different, very
compact extended formulations, see, e.g., \cite{Goemans14}. In the contrary, we show that the relaxation complexity of $
\PERM_n $ has exponential growth in $ n $:

\begin{theorem}
    The asymptotic growth of $ \rc(\PERM_n) $ is $ 2^{\Theta(n)} $.
\end{theorem}
\begin{proof}
    Let $ m := \lfloor \frac{n}{2} \rfloor $. For any set $ S \subseteq [n] $ with $ |S| = m $ select an integer vector
    $ x^S \in \Z^n $
    with $ \{ x^S_i : i \in S\} = \setdef{1,\dotsc,m-1} $ and $m-1$ 
    occurring twice among the $x^S_i$ ($i\in S$) and $ \{ x^S_i : i \in [n] \setminus S\} =
    \setdef{m+2,\dotsc,n} $ and $m+2$ occurring twice among the $x^S_i$ ($i\in  [n] \setminus S$).
    Such a vector is not contained in $ \conv(\PERM_n) $ as
    \[
        \sum_{i \in S} x_i^S = 1 + 2 + \dotsc + (|S| - 1) + (|S| - 1) < \frac{|S|(|S|+1)}{2}
    \]
    On the other hand, note that this is the only constraint from \eqref{eq:perm-ineq} that is violated by $ x^S $. In
    particular, $ x^S \in \aff(\PERM_n) $ holds.

    Let $ S_1, S_2 \subseteq [n] $ with $ |S_1| = |S_2| = m $ be distinct. We will show that $ x := \frac{1}{2} \cdot
    (x^{S_1} + x^{S_2} ) \in \conv(\PERM_n) $ holds. Since $ x $ satisfies all constraints that are satisfied by both
    $ x^{S_1} $ and $ x^{S_2} $, it suffices to show that $ \sum_{i \in T} x_i \geq \frac{|T|(|T|+1)}{2} $ holds for $ T
    \in \{ S_1, S_2 \} $. W.l.o.g. we may assume that $ T = S_1 $ and obtain
    \begin{align*}
        \sum_{i \in S_1} x_i & = \frac{1}{2} \sum_{i \in S_1} x_i^{S_1} + \frac{1}{2} \sum_{i \in S_1} x_i^{S_2} \\
        & = \frac{1}{2} \left( \frac{m(m+1)}{2} - 1 \right) + \frac{1}{2} \sum_{i \in S_1} x_i^{S_2} \\
        & \geq \frac{1}{2} \left( \frac{m(m+1)}{2} - 1 \right) + \frac{1}{2} \left( \frac{m(m+1)}{2} + 2 \right) \\
        & = \frac{m(m+1)}{2} + \frac{1}{2} \geq \frac{|T|(|T|+1)}{2}. \enspace \qed
    \end{align*}
    Thus, the set $ H := \setdef{x^S}[S \subseteq [n], \, |S| = m] $ is a hiding set for $ \PERM_n $. Our claim follows
    from Proposition \ref{prop:hiding-set} and the fact that $ |H| = \binom{n}{\lfloor \frac{n}{2} \rfloor} =
    2^{\Theta(n)} $. \qed
\end{proof}

\subsection{Parity}
\label{sec:parity}

\noindent
The final structural class we consider deals with the restriction that the number of selected elements of a given set
has a certain parity. Let us call a binary vector $ a \in \{0,1\}^d $ \emph{even} (\emph{odd}) if the sum of its entries
is even (odd). In \cite{Jeroslow75} it is shown that the number of inequalities needed to separate
\[
    \EVEN_n := \setdef{ x \in \{0,1\}^n }[ x \text{ is even}]
\]
from all other points in $ \{0,1\}^n $ is exactly $ 2^{n-1} $. This is done by showing that
\[
    \ODD_n := \setdef{ x \in \{0,1\}^n }[ x \text{ is odd}]
\]
is a hiding set for $ \EVEN_n $ (although the notion is different from ours). Hence, with Corollary \ref{cor:binary}, we
obtain:
\begin{theorem}
    \label{thm:parity}
    The asymptotic growth of $ \rc(\EVEN_n) $ is $ \Theta(2^n) $. \qed
\end{theorem}

\makeatletter{}\subsubsection{$ T $-joins}

\noindent
As a well-known representative of this structural class let us consider $ \TJOINS_n $, which is, for given $ T \subseteq
V_n $, defined as the set of characteristic vectors of $ T $-joins in the complete graph on $ n $ nodes. Let us recall
that a $ T $-join is a set $ J \subseteq E_n $ of edges such that $ T $ is equal to the set of nodes of odd degree in
the graph $ (V_n, J) $. Note, that if a $ T $-join exists, then $ |T| $ is even.

\begin{theorem}
    Let $ n $ be even and $ T \subseteq V_n $ with $ |T| $ even. Then, $ \rc(\TJOINS_n) \geq 2^{\frac{n}{4}-1} $.
\end{theorem}
\begin{proof}
    Since $ n $ is even and $ |T| $ is even, we may partition $ V_n $ into pairwise disjoint sets $ T_1, \, T_2, \, U_1,
    \, U_2 $ with $ T = T_1 \cup T_2 $, $ k = |T_1| = |T_2| $ and $ \ell = |U_1| = |U_2| $. Let $ M_1, \dotsc, M_k $ be
    pairwise edge-disjoint matchings of cardinality $ k $ that connect nodes from $ T_1 $ with nodes from $ T_2 $.
    Analogously, let $ N_1, \dotsc, N_\ell $ be pairwise edge-disjoint matchings of cardinality $ \ell $ that connect
    nodes from $ U_1 $ with nodes from $ U_2 $. For $ b \in \{0,1\}^k $ and $ c \in \{0,1\}^\ell $ let
    \[
        J(b,c) := \bigg( \bigcup_{i : b_i = 1} M_i \bigg) \cup \bigg( \bigcup_{j : c_j = 1} N_j \bigg) \subseteq E_n.
    \]
    By definition, $ J(b,c) $ is a $ T $-join if and only if $ b $ is odd and $ c $ is even. Let $ b^* \in \{0,1\}^k $
    odd and $ c^* \in \{0,1\}^\ell $ even be arbitrarily chosen but fixed. Since $ \ODD_n $ is a hiding set for $
    \EVEN_n $ and vice versa, it is now easy to see that both sets
    \begin{align*}
        H_1 & := \setdef{ J(b,c^*) }[b \in \{0,1\}^k \text{ even}] \\
        H_2 & := \setdef{ J(b^*,c) }[c \in \{0,1\}^\ell \text{ odd}],
    \end{align*}
    are hiding sets for $ \TJOINS_n $. Our claim follows from Proposition \ref{prop:hiding-set} and the fact that
    \begin{align*}
        \max \setdef{|H_1|, |H_2|} = \max \setdef{2^{k-1}, 2^{\ell-1}} & = \max \setdef{2^{\frac{1}{2}|T|-1},
        2^{\frac{1}{2}(n-|T|)-1}} \\
        & \geq 2^{\frac{1}{2} \cdot \frac{n}{2} - 1}. \enspace \qed
    \end{align*}
\end{proof}

\makeatletter{}\section{The Role of Rationality}
\label{sec:rationality}

\noindent
An interesting question is whether it may help (in order to construct a
relaxation having few facets) to use irrational coordinates in the description of a relaxation. First, let us show that
one does not lose too much when restricting to rational relaxations only:

\begin{proposition}
    \label{thm:rc-irrational}
    Let $ X \subseteq \Z^d $ be finite and $ \rc_{\Q}(X) $ be the smallest number of facets of any rational relaxation
    for $ X $. Then, $ \rc_{\Q}(X) \leq \rc(X) + \dim(X) + 1 $.
\end{proposition}
\begin{proof}
    Since $ X $ is finite, there exists a rational simplex $ \Delta \subseteq \R^d $ of dimension $ \dim(X) $ such that
    $ X \subseteq \Delta $. Let $ R $ be any relaxation of $ X $ having $ f $ facets and set $ B := (\Z^d \setminus X)
    \cap \Delta $. Since $ B \cap R = \emptyset $ and $ |B| < \infty $, we are able to slightly perturb the
    facet-defining inequalities of $ R $ in order to obtain a polyhedron $ \tilde{R} $ such that $ B \cap \tilde{R} =
    \emptyset $ and $ \tilde{R} $ is rational. Now $ \tilde{R} \cap \Delta $ is still a relaxation for $ X $, which is
    rational and has at most $ f + (\dim(\Delta) + 1) = f + \dim(X) + 1 $ facets. \qed
\end{proof}

\noindent
However, we are not aware of any polyhedral set $ X $ where $ \rc(X) < \rc_{\Q}(X) $. In fact, we even do not know if $
\rc(\Delta_d) < d + 1 $ holds. Note that any relaxation $ R $ for $ \Delta_d $ that has less than $
d+1 $ facets has to be unbounded. Hence, if $ R $ was rational, it would contain a rational ray and hence infinitely
many integer points, which shows $ \rc_{\Q}(\Delta_d) = d+1 $.

In order to show that any relaxation for $ \{0,1\}^d $ has at least $ d+1 $ facets, we basically used the fact that for
any line $ L(c) := \{ \lambda c \in \R^d : \lambda \in \R \} $ with $ c \in \R^d \setminus \{ \zerovec \} $, the set $
[0,1]^d + L(c) $ contains infinitely many integer points. Unfortunately, such a statement is not true for the general
simplex:

\begin{customexample}
    Consider the $ 5 $-dimensional simplex
    \[
        S := \conv \setdef{ \zerovec, \unitvec_1, \unitvec_2, \unitvec_3, \unitvec_1 + \unitvec_3 + \unitvec_4,
        \unitvec_2 + \unitvec_3 + \unitvec_5 } \subseteq \R^5.
    \]
    The polyhedron $ S + L((0,0,0,1,\sqrt{2})^T) $ does not contain any other integer points than those in $ S $.
    Indeed, let $ p + \lambda \cdot (0,0,0,1,\sqrt{2})^T $ be integral for some $ p \in S $ and some $ \lambda \in \R $.
    It is easy to see that $ p $ has to be one of the vertices of $ S $ and hence $ p $ is integral.
    Thus, $ \lambda $ and $ \lambda \sqrt{2} $ are both integral, which implies $ \lambda = 0 $.
\end{customexample}
Since $ S $ does not contain other integer points than its vertices, we can apply a unimodular transformation and
obtain a direction $ c' \in \R^5 $ such that $ R = \conv(\Delta_5) + L(c') $ is indeed an unbounded relaxation for $
\Delta_5 $. However, it can be verified that $ R $ has more than $ 6 $ facets in this case.

These simple observations lead to the following questions, whose answers are not known to the authors:
\begin{enumerate}
    \item Is it true that $ \rc(\Delta_d) = d+1 $ holds for all $ d $?
    \item Is it true that $ \rc(X) \geq \dim(X)+1 $ holds for all polyhedral (or at least finite) sets $ X \subseteq
    \R^d $?
    \item Is there any polyhedral (or even finite) set $ X \subset \Z^d $ such that $ \rc(X) < \rc_{\Q}(X) $?
\end{enumerate}
We close our paper by giving at least some non-constant lower bound on the relaxation complexity of $ \Delta_d $:
\begin{proposition}
    \label{prop:lower-bound-simplex}
    For all $ k \geq 1 $, we have $ \rc(\Delta_{k!}) \geq k $.
\end{proposition}
\begin{proof}
    Clearly, the claim is true for $ k = 1 $. Further, it is easy to see that $ \rc(\Delta_m) \leq \rc(\Delta_n) $ holds
    for all $ m \leq n $.

    Thus, by setting $ d(k) := k!-1 $, it suffices to show that even $ \rc(\Delta_{d(k)}) \geq k $ holds for all $ k
    \geq 2 $, which we will show by induction over $ k $. Note that the latter statement is true for $ k = 2 $. Let us
    assume that it is wrong for some $ k \geq 3 $, i.e., there exists a relaxation $ R \subseteq \R^d $ of $
    \Delta_{d(k)} $ that has $ \ell < k $ facets. Since $ \ell < k \leq d(k) = \dim(R) $, $ R $ is unbounded.

    We claim that every integer point $ p $ of $ \Delta_{d(k)} $ must lie in at least one facet of $ R $. Otherwise,
    since $ R $ is unbounded, there exist some $ \varepsilon > 0 $ and $ c \in \R^d \setminus \{\zerovec\} $ such that $
    p + L(c,0,\varepsilon) $ is contained in $ R $. By Corollary \ref{cor:minkowski}, $ L(c,0,\varepsilon) $ contains
    infinitely many integer points and so does $ R $, a contradiction.

    Hence, there must be a facet of $ R $ that contains $ t \geq \frac{d(k)+1}{l} $ vertices $ v_1,\dotsc,v_t $
    of $ \Delta_{d(k)} $. Let $ \mathcal{H} $ be the affine subspace spanned by $ v_1,\dotsc,v_t $ and let $ \varphi \colon
    \mathcal{H} \cap \Z^{d(k)} \to \Z^{t-1} $ be an affine isomorphism mapping $ \{ v_1,\dotsc,v_t \} $ to $
    \Delta_{t-1} $. Extending $ \varphi $ to an affine map from $ \mathcal{H} $ to $ \R^{t-1} $ yields that $ R' :=
    \varphi(R \cap \mathcal{H}) $ is a relaxation of $ \Delta_{t-1} $. Since
    \begin{align*}
        t-1 \geq \frac{d(k)+1}{l} - 1 \geq \frac{d(k)+1}{k} - 1 = \frac{k!-1+1}{k}-1 = (k-1)! - 1 = d(k-1),
    \end{align*}
    by induction, we must have that $ R' $ has at least $ k - 1 $ facets. On the other hand, note that $ R' $ has at
    most $ l-1 $ facets. This implies $ k \leq l $, a contradiction to our assumption. \qed
\end{proof}
 
\makeatletter{}\subsubsection*{Acknowledgements}
We would like to thank Gennadiy Averkov for valuable comments on this work.

\bibliographystyle{spmpsci}

\begin{thebibliography}{10}
\providecommand{\url}[1]{{#1}}
\providecommand{\urlprefix}{URL }
\expandafter\ifx\csname urlstyle\endcsname\relax
  \providecommand{\doi}[1]{DOI~\discretionary{}{}{}#1}\else
  \providecommand{\doi}{DOI~\discretionary{}{}{}\begingroup
  \urlstyle{rm}\Url}\fi

\bibitem{AroraB09}
Arora, S., Barak, B.: Computational complexity: a modern approach.
\newblock Cambridge University Press (2009)

\bibitem{AvisT13}
Avis, D., Tiwary, H.R.: On the extension complexity of combinatorial polytopes.
\newblock In: F.V. Fomin, R.~Freivalds, M.Z. Kwiatkowska, D.~Peleg (eds.)
  Automata, Languages, and Programming, \emph{Lecture Notes in Computer
  Science}, vol. 7965, pp. 57--68. Springer Berlin Heidelberg (2013).
\newblock \doi{10.1007/978-3-642-39206-1_6}.
\newblock \urlprefix\url{http://dx.doi.org/10.1007/978-3-642-39206-1_6}

\bibitem{FioriniMPTW12}
Fiorini, S., Massar, S., Pokutta, S., Tiwary, H.R., de~Wolf, R.: Linear vs.
  semidefinite extended formulations: exponential separation and strong lower
  bounds.
\newblock In: STOC, pp. 95--106 (2012)

\bibitem{GavishG78}
Gavish, B., Graves, S.C.: The travelling salesman problem and related problems.
\newblock Tech. rep., Operations Research Center, Massachusetts Institute of
  Technology (1978)

\bibitem{Goemans14}
Goemans, M.: Smallest compact formulation for the permutahedron ({To appear in
  a forthcoming issue of Mathematical Programming, Series B ``Lifts of Convex
  Sets'' edited by V. Kaibel and R. Thomas}) (2014).
\newblock \urlprefix\url{http://www-math.mit.edu/~goemans/publ.html}

\bibitem{Jeroslow75}
Jeroslow, R.: On defining sets of vertices of the hypercube by linear
  inequalities.
\newblock Discrete Mathematics \textbf{11}(2), 119--124 (1975)

\bibitem{KaibelW14}
Kaibel, V., Weltge, S.: Lower bounds on the sizes of integer programs without
  additional variables.
\newblock In: J.~Lee, J.~Vygen (eds.) Integer Programming and Combinatorial
  Optimization. Proceedings of IPCO XVII, Bonn, \emph{Lecture Notes in Computer
  Science}, vol. 8494. Springer (2014)

\bibitem{LeeM07}
Lee, J., Margot, F.: On a binary-encoded ilp coloring formulation.
\newblock INFORMS Journal on Computing \textbf{19}(3), 406--415 (2007)

\bibitem{Martin91}
Martin, R.K.: Using separation algorithms to generate mixed integer model
  reformulations.
\newblock Oper. Res. Lett. \textbf{10}(3), 119--128 (1991).
\newblock
  \urlprefix\url{http://www.sciencedirect.com/science/article/pii/016763779190028N}

\bibitem{MillerTZ60}
Miller, C.E., Tucker, A.W., Zemlin, R.A.: Integer programming formulation of
  traveling salesman problems.
\newblock J. ACM \textbf{7}(4), 326--329 (1960)

\bibitem{Minkowski96}
Minkowski, H.: Geometrie der Zahlen.
\newblock Teubner Verlag, Leipzig (1896)

\bibitem{PokuttaV13}
Pokutta, S., Vyve, M.V.: A note on the extension complexity of the knapsack
  polytope.
\newblock Oper. Res. Lett. \textbf{41}(4), 347--350 (2013).
\newblock \doi{http://dx.doi.org/10.1016/j.orl.2013.03.010}.
\newblock
  \urlprefix\url{http://www.sciencedirect.com/science/article/pii/S0167637713000394}

\bibitem{Rado52}
Rado, R.: An inequality.
\newblock J. London Math. Soc. \textbf{27}, 1--6 (1952)

\bibitem{Rothvoss14}
Rothvoss, T.: The matching polytope has exponential extension complexity.
\newblock In: Proceedings of the 46th Annual ACM Symposium on Theory of
  Computing, STOC '14, pp. 263--272. ACM, New York, NY, USA (2014).
\newblock \doi{10.1145/2591796.2591834}.
\newblock \urlprefix\url{http://doi.acm.org/10.1145/2591796.2591834}

\bibitem{Schrijver86}
Schrijver, A.: Theory of linear and integer programming.
\newblock John Wiley \& Sons, Inc., New York, NY, USA (1986)

\bibitem{Schrijver03}
Schrijver, A.: Combinatorial Optimization -- Polyhedra and Efficiency.
\newblock Springer (2003)

\bibitem{Yannakakis91}
Yannakakis, M.: Expressing combinatorial optimization problems by linear
  programs.
\newblock J. Comput. Syst. Sci. \textbf{43}(3), 441--466 (1991)

\end{thebibliography}

\end{document}